\documentclass[10pt,oneside,a4paper]{article} 	

\usepackage[fleqn]{amsmath} 				
\usepackage{amsthm,amsfonts,latexsym,amssymb,amscd} 	
\usepackage[mathscr]{eucal}   				
\usepackage[all]{xy}							
\usepackage[draft=false]{hyperref} 		
\usepackage{framed}

\newcommand{\B}{\mathcal{B}}

\renewcommand{\H}{\mathcal{H}} 	

\newcommand{\Cg}{\mathfrak{C}}

\newcommand{\Fg}{\mathfrak{F}}

\newcommand{\CC}{{\mathbb{C}}}

\newcommand{\NN}{{\mathbb{N}}}

\newcommand{\RR}{{\mathbb{R}}}

\newcommand{\As}{{\mathscr{A}}}

\newcommand{\Is}{{\mathscr{I}}}

\newcommand{\Ps}{{\mathscr{P}}}	
\newcommand{\Ss}{{\mathscr{S}}}

\DeclareFontFamily{U}{rsfs}{\skewchar\font127 }
\DeclareFontShape{U}{rsfs}{m}{n}{%
   <5> <6> rsfs5
   <7> rsfs7
   <8> <9> <10> <10.95> <12> <14.4> <17.28> <20.74> <24.88> rsfs10
}{}
\DeclareSymbolFont{rsfs}{U}{rsfs}{m}{n}
\DeclareSymbolFontAlphabet{\scr}{rsfs}

\newcommand{\Cf}{\scr{C}}

\newcommand{\Mf}{\scr{M}}
\newcommand{\Sf}{\scr{S}}

\DeclareMathOperator{\CCl}{\CC l}

\DeclareMathOperator{\spa}{span}
\DeclareMathOperator{\id}{Id}

\DeclareMathOperator{\dom}{Dom}

\DeclareMathOperator{\ev}{ev}

\DeclareMathOperator{\Ad}{Ad}

\DeclareMathOperator{\End}{End}

\renewcommand{\emph}{\textbf} 								
\newcommand{\mip}[2]{(#1\mid #2)}							
\newcommand{\ip}[2]{\langle #1\mid #2\rangle}			

\newcommand{\hlink}[2]{\href{#1}{\texttt{#2}}} 		

\newtheorem{theorem}{Theorem}[section]					

\newtheorem{lemma}[theorem]{Lemma}
\newtheorem{proposition}[theorem]{Proposition}

\newtheorem{definition}[theorem]{Definition}
\newtheorem{remark}[theorem]{Remark}

\numberwithin{equation}{section}  	
\title{\textbf{A Remark on Gel'fand Duality for \\ Spectral Triples\footnote{Published in: Bull.~Korean~Math.~Soc.~48 (2011) n.~3:505-521.}}}

\author{\normalsize Paolo Bertozzini$^a$,  
Roberto Conti$^b$\footnote{Current address: Dipartimento di Scienze, Universit\`a di Chieti-Pescara ``G. D'Annunzio'', Viale Pindaro 42, I-65127 Pescara, Italy.}, 
Wicharn Lewkeeratiyutkul$^c$
\\
\normalsize  \textit{$^a$Department of Mathematics and Statistics, Faculty of Science and Technology}
\\
\normalsize \textit{Thammasat University, Bangkok 12121, Thailand}
\\
\normalsize e-mail: \texttt{paolo.th@gmail.com}
\\
\normalsize \textit{$^b$Mathematics, School of Mathematical and Physical Sciences,} 
\\
\normalsize \textit{University of Newcastle, Callaghan, NSW 2308, Australia}  
\\ 
\normalsize e-mail: \texttt{conti@sci.unich.it}
\\ 
\normalsize \textit{$^c$Department of Mathematics and Computer Science, }
\\
\normalsize \textit{Faculty of Science, Chulalongkorn University, Bangkok 10330, Thailand}
\\
\normalsize e-mail: \texttt{Wicharn.L@chula.ac.th}
}

\date{\normalsize original arXiv version: 18 December 2008, 
published version: 11 March 2011, 
\\
new arXiv version: 26 December 2011}

\begin{document}

\maketitle

\begin{abstract} \noindent
We present a duality between the category of compact Riemannian spin manifolds (equipped with a given spin bundle and charge conjugation) with isometries as morphisms and a suitable ``metric'' category of spectral triples over commutative 
pre-\hbox{C*-algebras}. 
We also construct an embedding of a ``quotient'' of the category of spectral triples introduced in~\cite{BCL3} into the latter metric category. 
Finally we discuss a further related duality in the case of orientation and spin-preserving maps between manifolds of fixed dimension. 

\smallskip


\smallskip

\noindent
\emph{MSC-2010:} 			
					46L87,			
					46M15, 			
					18F99, 			
					15A66.			
\end{abstract}

\section{Introduction}

Although the main strength of non-commutative geometry is a full treatment of non-commutative algebras as ``duals of geometric spaces'', the foundation of the theory relies on the construction of suitable categorical equivalences, resp.~anti-equivalences (i.e.~covariant, resp.~contravariant functors that are isomorphisms of categories ``up to natural transformations'') between categories of ``geometric spaces'' and categories of commutative algebras of functions over these spaces, or some closely related structures. (For the elementary background in ``category theory'' the reader can refer to the on-line introduction by 
J.~Baez~\cite{B} and the classical books by S.~McLane~\cite{M} and M.~Barr-C.~Wells~\cite{BW}.)

Typical examples of such (anti-)equivalences 
are listed below, itemized by the name of the people who worked them out: 
\begin{itemize}
\item
\emph{Hilbert:} 
between algebraic sets and finitely generated algebras over an algebraically closed field~\cite{H};   
\item
\emph{Stone:}
between totally disconnected compact Hausdorff topological spaces and Boolean algebras~\cite{St1,St2}; 
\item
\emph{Gel'fand-Na\u\i mark:} 
between the category of continuous maps of compact Hausdorff topological spaces and the category of unital involutive homomorphisms of unital commutative C*-algebras~\cite{G2,GN}; 
\item
\emph{Halmos-von Neumann:}
between the category of measurable maps of measure spaces and the category of unital involutive homomorphisms of commutative von Neumann algebras~\cite{HvN}; 
\item
\emph{Serre-Swan:} between the category of vector bundle maps of finite-dimensional locally trivial 
vector bundles over a compact Hausdorff topological space and the category of homomorphisms of finite projective modules over a commutative unital C*-algebra~\cite{Se,Sw}; 
\item
\emph{Cartier-Grothendieck:} 
between the category of commutative schemes (ringed spaces) in algebraic geometry and the category of topoi (sheaves over topological spaces); see I.~Dolgachev's historical notes ~\cite[Section~1]{D};
\item
\emph{Takahashi:}
between the category of Hilbert bundles on (different) compact Hausdorff spaces and the category of Hilbert C*-modules over (different) commutative unital C*-algebras~\cite{Ta1,Ta2};
\end{itemize}

Even more dualities arise when the spaces in question are equipped with additional structure, most notably a group structure
or the like (see Pontryagin-Van Kampen~\cite{Po,VK}, Tannaka-Kre\u\i n~\cite{Ta,Kr} and Doplicher-Roberts~\cite{DR}).

In this paper we will focus our attention on the Gel'fand-Na\u\i mark duality, to which the other dualities are related in significant way.
In short, the fundamental message that can be read off from the celebrated Gel'fand-Na\u\i mark theorem on commutative 
C*-algebras is that, at the ``topological level'', the information on a ``space'' can be completely encoded in (and recovered from) a 
suitable ``algebraic structure''. 

In applications to physics (at least for those branches that are dealing with ``metric structures'' such as general relativity), it would be important to ``tune'' Gel'fand-Na\u\i mark's correspondence in order to embrace classes of spaces with more detailed geometric structures (e.g.~differential, metric, connection, curvature).

In recent times, Connes' non-commutative geometry~\cite{C,FGV} has emerged as the most outstanding proposal in this direction, based on the notion of spectral triple.

In this note we provide a further example of categorical anti-equivalence between Riemannian spin-manifolds and 
commutative Connes' spectral triples (see theorem~\ref{th: co-re}). 
This line of thought is expected to play an important role in future developments of the categorical structure of non-commutative geometry, and spectral triples in particular (see~\cite{BCL0}), as well as in the study of (geometric) quantization, where the construction of functorial relations between ``commutative'' and ``quantum'' spaces are central points of investigation. 

Although the idea of reconstructing a smooth manifold out of a commutative spectral triple has been latent for some time
(see~\cite{C4,C10,Re1,RV,C17,C12}), the point to promote it to a categorical level seems to be new. 
Our main tool is the notion of metric morphisms of spectral triples, namely those preserving Connes' distance on the state space.

In the second part of the paper, we examine some connection between the category of ``metric spectral triples'' (on which the equivalence result is based) and our previous work on morphism of spectral triples~\cite{BCL3}. 
It should be possible to provide other equivalence results in terms of categories of spectral triples based on different notions of morphism (at least for some classes of Riemannian manifolds); some of these issues are presently under investigation 
(see~\cite[Section~4.1]{BCL0} for an overview). 

\smallskip

It should be remarked that Connes' distance formula has been systematically adopted by M.~Rieffel as the backbone of his notion of quantum compact metric space (see~\cite{Ri} and references therein).
Although we present our result in the framework of Connes' spectral triples, it is likely that our ideas might find some application also in Rieffel's framework. 

\smallskip

In order to keep the length of this paper as short as possible, we will refer to the literature for all the background material and only recall the basic definitions. 

\subsection{Spectral Triples}

Following A.~Connes' axiomatization (see~\cite{C, FGV,C17} for all the details), a \emph{compact spectral triple} $(\As,\H,D)$ consists of 
\begin{itemize}
\item[a)] 
a unital pre-C*-algebra $\As$ (that is sometimes required to be closed under holomorphic functional calculus), 
\item[b)] 
a (faithful) representation $\pi: \As \to \B(\H)$ of the algebra $\As$ on a Hilbert space $\H$ and 
\item[c)] 
a \emph{Dirac operator}, i.e.~a (generally unbounded) self-adjoint operator $D$,  
with compact resolvent $(D-\lambda)^{-1}$ for every $\lambda \in \CC - \RR$ and such that $[D,\pi(a)]_{-}\in \B(\H)$  
for all $a \in \As$, where $[x,y]_\pm:= xy\pm yx$ denote the anticommutator and the commutator, respectively, of $x,y \in \B(\H)$.
\end{itemize}

A spectral triple is called \emph{even} if it is equipped with a grading operator, i.e.~a bounded self-adjoint operator 
$\Gamma \in \B(\H)$ such that:
\begin{gather*}
\Gamma^2=\text{Id}_\H; \quad  [\Gamma, \pi(a)]_{-}=0 \quad \forall a \in \As; \quad
[\Gamma, D]_{+}=0. 
\end{gather*}
A spectral triple without grading is called \emph{odd}. 

\medskip 

A spectral triple is \emph{regular} if the functions $\Xi_x: t\mapsto \exp(it|D|)x\exp(-it|D|)$ are ``smooth'' 
i.e.~$\Xi_x\in \text{C}^\infty(\RR, \B(\H))$ for every $x \in \Omega_D(\As)$, 
where we define 
\begin{equation*}
\Omega_D(\As):=\spa\{\pi(a_0)[D,\pi(a_1)]_- \cdots [D,\pi(a_n)]_- \  |  \  n\in \NN, \ a_0, \dots, a_n \in \As\}.
\end{equation*}
This regularity condition can be equivalently expressed requiring that, for all $a\in \As$, $\pi(a)$ and $[D,\pi(a)]_-$ are contained in 
$\cap_{m=1}^\infty \text{Dom}\, \delta^m$, where $\delta$ is the derivation given by $\delta(x):=[|D|, x]_{-}$. 

\medskip 

The spectral triple is \emph{$n$-dimensional} iff there exists an integer $n$ such that the Dixmier trace of $|D|^{-n}$ is finite 
non-zero. 
A spectral triple is \emph{$\theta$-summable} if $\exp(-tD^2)$ is a trace-class operator for every $t>0.$

A spectral triple is \emph{real} if it is equipped with a real structure i.e.~an antiunitary operator $J: \H \to \H$ such that: 
\begin{gather*}
[\pi(a), J\pi(b^*)J^{-1}]_{-}=0 \quad \forall a,b \in \As; \\
[[D, \pi(a)]_{-}, J\pi(b^*)J^{-1}]_{-}=0 \quad \forall a,b \in \As, \quad {\emph{first-order condition};} \\ 
J^2=\pm\text{Id}_\H;  \quad [J,D]_{\pm}=0; \\
\text{and, only in the even case,} \quad
[J,\Gamma]_{\pm}=0,
\end{gather*}
where the choice of $\pm$ in the last three formulas depends on the ``dimension'' $n$ of the spectral triple modulo $8$ according to the following table: 
\begin{center}\label{tb: J}
\begin{tabular}{|l|c|c|c|c|c|c|c|c|}
\hline
$n$									&$0$	&$1$	&$2$	&$3$	&$4$	&$5$	&$6$	&$7$	\\
	\hline
$J^2=\pm\text{Id}_\H$	&	$+$	&	$+$		&	$-$		&	$-$		&	$-$		&	$-$		&	$+$		& $+$		\\
	\hline
$[J,D]_{\pm}=0$				&	$-$	&	$+$		&	$-$		&	$-$		&	$-$		&	$+$		&	$-$		& $-$		\\
\hline
$[J,\Gamma]_{\pm}=0$	&	$-$	&			&	$+$		&			&	$-$		&			&	$+$		& 		\\
\hline
\end{tabular}
\end{center}

\medskip

A spectral triple is \emph{finite} if $\H_\infty := \cap_{k=1}^\infty \text{Dom}\, D^k$ is a finite projective $\As$-bimodule 
and \emph{absolutely continuous} if, there exists a Hermitian form $(\xi,\eta)\mapsto\mip{\xi}{\eta}$ on $\H_\infty$ such that, for all $a\in \As$, 
$\ip{\xi}{\pi(a)\eta}$ is the Dixmier trace of $\pi(a)\mip{\xi}{\eta}|D|^{-n}$.

\medskip

An $n$-dimensional spectral triple is said to be \emph{orientable} if there is a Hochschild cycle
$c=\sum_{j=1}^m a^{(j)}_0\otimes a^{(j)}_1\otimes \cdots \otimes a^{(j)}_n$ such that its ``representation'' on the Hilbert space $\H$, $\pi(c)=\sum_{j=1}^m\pi(a^{(j)}_0)[D, \pi(a^{(j)}_1)]_-\cdots[D,\pi(a^{(j)}_n)]_-$ is the grading operator in the even case or the identity operator in the odd case. (In the following, in order to simplify the discussion, we will always refer to a ``grading operator'' $\Gamma$ that actually coincides with the grading operator in the even case and that is by definition the identity operator in the odd case.)

\medskip

A real spectral triple is said to satisfy \emph{Poincar\'e duality} if its fundamental class in the KR-homology of 
$\As\otimes \As^{\rm op}$ induces (via Kasparov intersection product) an isomorphism between the K-theory K$_\bullet(\As)$ and the K-homology K$^\bullet (\As)$ of $\As$. In~\cite{RV} some of the axioms are reformulated in a different form, in particular this condition is replaced by the requirement that the C*-module completion of $\H_\infty$ is a Morita equivalence bimodule between (the norm completions of) $\As$ and $\Omega_D(\As)$.

\medskip

A spectral triple will be called \emph{commutative}, or \emph{Abelian}, whenever $\As$ is commutative. 

\medskip 

Finally a spectral triple with real structure $J$ and grading $\Gamma$ is \emph{irreducible} if there is no non-trivial closed subspace in $\H$ that is invariant for 
$\pi(\As), D, J, \, \Gamma$.

\subsection{Reconstruction Theorem (Commutative Case)}

Let $M$ be a real compact orientable Riemannian $m$-dimensional spin C$^\infty$ manifold with a given volume form $\mu_M$. 
Let us denote (see~\cite{S} for details) by $S(M)$ a given irreducible complex spinor bundle over $M$ i.e.~a bundle over $M$ equipped with a left action $c: \CCl^{(+)}(T(M))\otimes S(M)\to S(M)$ of the ``Clifford'' 
bundle $\CCl^{(+)}(T(M))$ inducing a bundle isomorphism between $\CCl^{(+)}(T(M))$ and $\End(S(M))$, where
following~\cite[Page~373]{FGV} we denote by $\CCl^{(+)}(T(M))$ the complexified Clifford bundle of $M$ if 
$\dim M$ is even and respectively its even subalgebra bundle $\CCl^+(T(M))$ if $\dim M$ is odd.
Let $[S(M)]$ be the spin$^c$ structure of $M$ determined by $S(M)$. 

\medskip

Recall that an orientable Riemannian manifold is spin$^c$ if it admits a complex irreducible spinor bundle~\cite[Definition~7]{S}. A spin$^c$ manifold usually admits several inequivalent spin$^c$ structures and that for a given spin$^c$ structure, a complex irreducible spinor bundle over $M$ is determined only up to 
(Hermitian) bundle isomorphism. A spin$^c$ manifold is spin if and only if it admits a complex spinor bundle with a charge conjugation~\cite[Definition~8]{S}. A spin manifold usually admits several inequivalent spin structures even for the same spin$^c$ structure and that for a given spin structure a conjugation operator is determined only up to intertwining with (Hermitian) bundle isomorphisms.

\medskip

Let $C_M$ be a given ``spinorial'' charge conjugation on $S(M)$ i.e.~an antilinear Hermitian bundle morphism such that 
$C_M\circ C_M=\pm\id_{S(M)}$ (signs depending on $\dim M$ modulo $8$ as in the table in section~\ref{tb: J}) that is ``compatible'' with the charge conjugation $\kappa$ in $\CCl^{(+)}(T(M))$, which is the composition of the natural grading operator and the canonical conjugation,  i.e.~$C_M(\beta(p)\cdot \sigma(p))=\kappa(\beta(p))\cdot C_M(\sigma(p))$, for any section 
$\beta \in \Gamma(\CCl^{(+)}(T(M)))$ of the Clifford bundle and any section $\sigma\in \Gamma(S(M))$ of the spinor bundle. 
We denote by $[(S(M),C_M)]$ the spin structure on $M$ determined by $C_M$.

\medskip

Let $\As_M:=$C$^\infty(M; \CC)$ be the commutative pre-C*-algebra of smooth complex valued functions on $M$. We denote by $\pi_M$ its representation by pointwise multiplication on the space $\H_M:=\text{L}^2(M,S(M))$, the completion of the space $\Gamma^\infty(M,S(M))$ of smooth sections of the spinor bundle $S(M)$ equipped with the inner product 
$\ip{\sigma}{\tau}:=\int_M\ip{\sigma(p)}{\tau(p)}_p\,\text{d}\mu_M$, where $\ip{\cdot}{\cdot}_p$ is the unique inner product on $S_p(M)$ compatible with the Clifford action and the Clifford product.  
Note that the spinorial charge conjugation $C_M$ (being unitary on the fibers) has a unique antilinear unitary extension 
$J_M:\H_M\to \H_M$ determined by $(J_M\sigma)(p):=C_M(\sigma(p))$ for $\sigma\in \Gamma^\infty(S(M))$ and $p\in M$.

\medskip

Let $\Gamma_M$ be the unique unitary extension on $\H_M$ of the operator $\Lambda_M$ on $\Gamma(S(M))$ acting by left action of the chirality element $\gamma\in\Gamma(\CCl^{(+)}(T(M)))$, that implements the grading $\chi$ of 
$\Gamma(\CCl^{(+)}(T(M)))$ as inner automorphism (the grading is actually the identity in odd dimension).

\medskip

Denote by $D_M$ the Atiyah-Singer Dirac operator on the Hilbert space $\H_M$, i.e.~the closure of the operator that on 
$\Gamma^\infty(S(M))$ is obtained by ``contracting'' the unique spin covariant derivative $\nabla^S$ (induced on 
$\Gamma^\infty(S(M))$ by the Levi-Civita covariant derivative of $M$, see~\cite[Theorem~9.8]{FGV}) with the Clifford multiplication. For a detailed discussion on Atiyah-Singer Dirac operators we refer to~\cite{BGV, LM, S}. 

\medskip

We have the following fundamental results: 
\begin{theorem}[Connes, see e.g.~\cite{C,C3} and Section~11.1 in~\cite{FGV}]\label{th: Co} 
 Given an orientable compact spin Riemannian $m$-dimensional differentiable manifold $M$, with a given complex spinor bundle 
$S(M)$, a given spinorial charge conjugation $C_M$ and a given volume form $\mu_M$, the data $(\As_M,\H_M,D_M)$ defines a commutative regular finite absolutely continuous $m$-dimensional spectral triple that is real, with real structure $J_M$, orientable, with grading 
$\Gamma_M$, and satisfies Poincar\'e duality. 
\end{theorem}

\begin{theorem}[Connes~\cite{C4,C17}]\label{th: Co-Re}
Let $(\As,\H,D)$ be an irreducible commutative real (with real structure $J$ and grading $\Gamma$) strongly regular (in the sense of~\cite[Definition~6.1]{C17}) $m$-dimensional finite absolutely continuous orientable spectral triple, with totally antisymmetric (in the last $m$ entries) Hochschild cycle, and satisfying Poincar\'e duality.
The spectrum of (the norm closure of) $\As$ can be endowed, in a unique way, with the structure of an $m$-dimensional
connected compact orientable spin Riemannian manifold $M$ with an
irreducible complex spinor bundle $S(M)$, a charge conjugation $J_M$ and a grading $\Gamma_M$ such that:
\begin{equation*}
\As\simeq C^\infty(M; \CC), \quad \H\simeq\text{L}^2(M,S(M)), \quad D \simeq D_M, \quad J \simeq J_M, \quad 
\Gamma \simeq \Gamma_M.
\end{equation*}
\end{theorem}
A.~Connes first proved the previous theorem~\ref{th: Co-Re}  under the additional condition that $\As$ is already given as the algebra of smooth complex-valued functions over a differentiable manifold $M$, namely $\As = C^\infty(M;\CC)$ 
(for a detailed proof see e.g.~\cite[Theorem~11.2]{FGV}), and 
conjectured~\cite{C4},~\cite[Theorem~6, Remark~(a)]{C10} the result for general commutative pre-C*-algebras $\As$.

A tentative proof of this last fact has been published by A.~Rennie~\cite{Re1}; some gaps were pointed out in the original argument, a different revised, but still incorrect, proof appears in~\cite{RV} (see also~\cite{RV2}) under some additional technical conditions. Recently A.~Connes~\cite{C17} finally provided the missing steps in the proof of the result. 

As a consequence, there exists a one-to-one correspondence between unitary equivalence classes of spectral triples and connected compact oriented Riemannian spin-manifolds up to spin-preserving isometric diffeomorphisms.

Similar results should also be available for spin$^c$ manifolds~\cite[Theorem~6, Remark~(e)]{C10}.

\subsection{Connes' Distance Formula}

Given a spectral triple $(\As,\H,D)$, let us denote by $\Ss(\As)$ and $\Ps(\As)$ the sets of \emph{states} and \emph{pure states} of the pre-C*-algebra $\As$, respectively. 
If $\As:=\text{C}^\infty(M;\CC)$, for all $p\in M$ we denote by $\ev_p: x\mapsto x(p)$ the ``evaluation functional'' in $p$ of the functions $x\in \As$ and note that $\ev_p\in \Ps(\As)$. Actually in this case $\Ps(\As)$ coincides with the set of all evaluation functionals. 

Going back to the general case, the \emph{Connes' distance} $d_D$ on $\Ps(\As)$ is the function on $\Ps(\As)\times\Ps(\As)$ given by 
\begin{equation*}
d_D(\omega_1,\omega_2):=\sup\{|\omega_1(x)-\omega_2(x)| \ | \ x \in \As, \ \|[D,\pi(x)]\|\leq 1\}.
\end{equation*}
Strictly speaking, without imposing other conditions, $d_D$ could also take the value $+\infty$ as in the case of non-connected manifolds.
In turn, one can use the same formula to define a ``distance'' on the set of all the states of $\As$.

\begin{theorem}[Connes's distance formula]~\cite[Proposition~9.12]{FGV} 
\label{th: Co-geo} 
If the spectral triple $(\As,\H,D)$ is obtained as in theorem~\ref{th: Co} from a compact finite-dimensional oriented Riemannian spin manifold $M$ equipped with a spinor bundle $S(M)$ and a spinorial charge conjugation $C_M$, then for every $p,q\in M$, 
$d_D(\ev_p,\ev_q)$ coincides with the geodesic distance
\begin{equation*}
d_M(p,q):=\inf\Big\{\int_a^b \|\gamma'(t)\|\, \text{d}t \ | \ \text{$\gamma$ is a geodesic with  $\gamma(a)=p$, $\gamma(b)=q$} \Big\}.
\end{equation*} 
\end{theorem}

Of course, given a unital $*$-morphism $\phi: \As_1\to \As_2$ there is a pull-back 
\hbox{$\phi^\bullet: \Ss(\As_2)\to \Ss(\As_1)$} 
defined by $\phi^\bullet(\omega):=\omega\circ\phi$ for all $\omega\in \Ss(\As_2)$.

\section{A Metric Category of Spectral Triples}

The objects of all of our categories will be compact spectral triples $(\As,\H,D)$, possibly with additional structure.  

Given two spectral triples $(\As_j,\H_j,D_j)$, with $j=1,2$, a \emph{metric morphism} of spectral triples 
$(\As_1,\H_1,D_1)\xrightarrow{\phi} (\As_2,\H_2,D_2)$ is by definition a unital epimorphism
$\phi:\As_1\to\As_2$ of pre-C*-algebras whose pull-back $\phi^\bullet:\Ps(\As_2)\to\Ps(\As_1)$ is an isometry
(note that if $\phi$ is an epimorphism, its pull-back $\phi^\bullet$ maps pure states into pure states), i.e.
\begin{equation*}
d_{D_1}(\phi^\bullet(\omega_1),\phi^\bullet(\omega_2))=d_{D_2}(\omega_1,\omega_2), \quad 
\forall \omega_1,\omega_2\in \Ps(\As_2).
\end{equation*}

Spectral triples with metric morphisms form a category $\Sf^m$. 

\begin{remark} \label{rm: eq}
A unitary equivalence of spectral triples gives an isomorphism in the metric category $\Sf^m$.
\end{remark}

\subsection{A Local Metric Category of Spectral Triples}\label{sec: sf}

For convenience of the reader, we recall here the definitions of morphisms of spectral triples proposed in our previous 
work~\cite[Sections~2.2-2.3]{BCL3}.

A \emph{morphism} in the category $\Sf$, between spectral triples $(\As_j,\H_j,D_j)$, $j=1,2$, is a pair $(\phi,\Phi)$, 
where $\phi: \As_1\to \As_2$ is a $*$-morphism between the pre-C*-algebras $\As_1,\As_2$ and $\Phi: \H_1\to \H_2$ is a bounded linear map in $\B(\H_1,\H_2)$ such that $\pi_2(\phi(x))\circ \Phi=\Phi\circ \pi_1(x)$, $\forall x \in \As_1$ and 
$D_2\circ\Phi(\xi)=\Phi \circ D_1(\xi)$ $\forall\xi\in \dom D_1$.

In a similar way, a \emph{morphism of real spectral triples} $(\As_j,\H_j,D_j,J_j)$ with $j=1,2$, in the category of real spectral triples $\Sf_r$, is a morphism in $\Sf$ such that $\Phi$ also satisfies $J_2\circ \Phi=\Phi \circ J_1$. 
Finally a \emph{morphism of even spectral triples} $(\As_j,\H_j,D_j,\Gamma_j)$ with $j=1,2$, in the category of even spectral triples 
$\Sf_e$, is a morphism in $\Sf$ such that $\Gamma_2\circ \Phi=\Phi \circ \Gamma_1$.
We will denote by $\Sf_I$ (respectively $\Sf_{Ir},\Sf_{Ire}$) the subcategory of $\Sf$ (respectively $\Sf_{r},\Sf_{re}$) consisting of ``isometric'' morphisms of spectral triples, i.e.~pairs $(\phi,\Phi)$ with $\phi$ surjective and $\Phi$ co-isometric. 

The reader can consult appendix~\ref{a: st} for a list of all the categories of spectral triples that are introduced and used in this paper. 

\section{The Metric Functor $\Cg$}

Let us consider the class $\Mf$ of C$^\infty$ metric isometries of compact finite-dimensional C$^\infty$ orientable Riemannian spin manifolds $M$ equipped with 
a fixed spinor bundle $S(M)$, a given spinorial charge conjugation $C_M$ and a volume form $\mu_M$. 
Note that in general a Riemannian isometry is not necessarily a  metric isometry.
The class $\Mf$ with the usual composition of functions forms a category. 
\begin{proposition}\label{prop: cr}
There is a contravariant functor $\Cg$ from the category $\Mf$ to the category $\Sf^m$ that to every triple $(M,S(M),C_M)\in \Mf$ associates the spectral triple $(\As,\H,D)\in \Sf^m$ given as in theorem~\ref{th: Co} and that to every smooth metric isometry 
$f: M_1\to M_2$ associates its pull-back $f^\bullet: \As_2\to\As_1$.  
\end{proposition}
\begin{proof}
Every smooth metric isometry $f: M_1\to M_2$ in $\Mf$ is a Riemannian isometry of $M_1$ onto a closed embedded submanifold $f(M_1)$ of $M_2$. Since every smooth function on a  closed embedded submanifold is the restriction of a smooth function on $M_2$, the pull-back $\phi:=f^\bullet$ is a unital epimorphism of the pre-C*-algebras $\phi: \As_2\to \As_1$ and, by 
theorem~\ref{th: Co-geo}, $\phi^\bullet: \Ps(\As_1)\to\Ps(\As_2)$ is metric-preserving:
\begin{align*}
d_{D_2}(\phi^\bullet(\omega_1),\phi^\bullet(\omega_2))
&=d_{D_2}(\phi^\bullet(\ev_p),\phi^\bullet(\ev_q)) =d_{D_2}(\ev_{f(p)},\ev_{f(q)})\\
&=d_{M_2}(f(p),f(q))=d_{M_1}(p,q)=d_{D_1}(\ev_p,\ev_q)\\ &=d_{D_1}(\omega_1,\omega_2), 
\end{align*}
where $p,q\in M_1$ are the unique points such that $\omega_1=\ev_p$ and $\omega_2=\ev_q$. 

Of course $\Cg(g\circ f)=(g\circ f)^\bullet=f^\bullet\circ g^\bullet=\Cg_f\circ\Cg_g$ and 
$\Cg_{\iota_M}=\iota_{\Cg(M)}$.
\end{proof}

\begin{definition}
An \emph{Atiyah-Singer spectral triple} is a spectral triple $(\As,\H,D)$ that satisfies all the conditions in A.Connes' reconstruction theorem~\ref{th: Co-Re} i.e.~$(\As,\H,D)$ is irreducible commutative real (with real structure $J$ and grading $\Gamma$) strongly regular (in the sense of~\cite[Definition~6.1]{C17}) $m$-dimensional finite absolutely continuous orientable, with totally antisymmetric (in the last $m$ entries) Hochschild cycle, and satisfies Poincar\'e duality.
\end{definition}
By the reconstruction theorem~\ref{th: Co-Re} every Atiyah-Singer spectral triple is isomorphic to a spectral triple constructed, as in theorem~\ref{th: Co}, from an orientable connected compact spin Riemannian finite-dimensional differentiable manifold equipped with a given complex spinor bundle and a spinorial charge conjugation.  

\medskip 

We denote by AS-$\Sf^m$ the full subcategory of the metric category $\Sf^m$ whose objects are 
direct sums of 
Atiyah-Singer spectral triples. 
In a completely similar way, we denote by AS-$\Sf$ the full subcategory of the local metric category $\Sf$ whose objects are direct sums of 
Atiyah-Singer spectral triples in $\Sf$.
The functor $\Cg$ in proposition~\ref{prop: cr} takes actually values in the category AS-$\Sf^m$. 

\medskip

Here we present the main result of this paper. 

\begin{theorem}\label{th: co-re}
The metric functor $\Cg$ is an anti-equivalence between the categories $\Mf$ and AS-$\Sf^m$.
\end{theorem}
\begin{proof}
The functor $\Cg$ is faithful: if $\Cg_f=\Cg_g$ for two smooth isometries $f,g: M_1\to M_2$, then $f^\bullet=g^\bullet$ as morphisms of pre-C*-algebras and hence they coincide also when uniquely extend to morphisms of C*-algebras of continuous functions and the result $f=g$ follows from Gel'fand duality theorem. 

The functor $\Cg$ is full: if $\phi: \Cg(M_2)\to\Cg(M_1)$ is a metric morphism in $\Sf^m$, as a homomorphisms of pre-C* algebras of smooth functions, $\phi$ extends uniquely to a morphism of C*-algebras of continuous functions and, from Gel'fand duality theorem, there exists a unique continuous function $f: M_1\to M_2$ such that $f^\bullet=\phi$. 
From the fact that $f^\bullet$ maps smooth functions on $M_2$ to smooth functions on $M_1$ it follows that $f$ is a smooth function between manifolds.
Since $\phi$ also preserves the spectral distances, it follows that $f$ is a smooth metric isometry hence a Riemannian isometry. 

The functor $\Cg$ is representative: for when restricted to the subcategory of connected manifolds with target the subcategory of irreducible spectral triples, this is actually a restatement of the reconstruction theorem~\ref{th: Co-Re} and remark~\ref{rm: eq}. 
Since the functor $\Cg$ maps disjoint unions of connected components into direct sums of spectral triples, the result follows.
\end{proof}

Unfortunately, at this stage, we cannot present a statement involving the category of all Abelian spectral triples. The above result raises naturally the issue of decomposing (Abelian) spectral triples in terms of irreducible components.

\begin{remark}\label{rk: co-re}
In restriction to the subcategory $\Mf_d$ of \emph{dimension-preserving} smooth isometries (i.e.~isometric immersions with fiberwise isomorphic tangent maps), the metric functor $\Cg$ is an anti-equivalence between $\Mf_d$ and the subcategory 
AS-$\Sf^m_d$ of metric morphisms of 
direct sums of irreducible 
Abelian spectral triples with the same dimension. 
In a similar way, denoting by $\Is(\Cf)$ 
the groupoid of isomorphisms of $\Cf$, we have that $\Cg|_{\Is(\Mf)}$ is an anti-equivalence between $\Is(\Mf)$ and the 
groupoid 
$\Is(\text{AS-}\Sf^m)$.
The groupoid of isomorphisms of 
$\Mf$ (always a subcategory of $\Mf_d$) is actually the ``disjoint union'' of denumerable ``connected components'' consisting of the categories of smooth bijective isometries of \hbox{$n$-dimensional} spin-manifolds.
\end{remark}

\section{Metric and Spin Categories}

We now proceed to establish a connection between the category $\Sf^m$ of metric spectral triples and the categories of spectral triples $\Sf$ (respectively real spectral triples $\Sf_r$) introduced in~\cite[Section~2.2-2.3]{BCL3} and briefly recalled in 
section~\ref{sec: sf}. 

Denote by $\Sf^0$ (respectively $\Sf^0_{Ired}$) the category of spectral triples whose morphisms are those homomorphisms of algebras $\phi$ for which there exists at least one $\Phi$ such that the pair $(\phi,\Phi)$ is a morphism in $\Sf$ (respectively 
$\Sf_{Ired}$).\footnote{Please refer to appendix~\ref{a: st} for the list of the categories of spectral triples.}
We have a ``forgetful'' full functor $\Fg:\Sf\to \Sf^0$ that to every morphism $(\phi,\Phi)$ in $\Sf$ associates $\phi$ as a morphism in $\Sf^0$. 

\begin{lemma}\label{lem: smooth}
A metric isometry of Riemannian manifolds with the same dimension is a smooth Riemannian isometry onto a union of connected components.
\end{lemma}
\begin{proof}
Let $f: M \to N$ be a metric isometry.  
Since $\dim M=\dim N$, by Brouwer's theorem, we see that $f$ is open and maps each connected component of $M$ onto a unique connected component of $N$. 
By the Myers-Steenrod theorem (see for example~\cite[Section~5.9, Theorem~9.1]{P}), any such bijective map between connected components is a smooth Riemannian surjective isometry; hence $f: M\to N$ is a smooth Riemannian isometry 
onto $f(M)$, a union of connected components of $N$. 
\end{proof}

Let $f: (M,S(M),C_M)\to (N,S(N),C_N)$ be a morphism in $\Mf_d$. 
Thanks to the last lemma, we can consider the differential $Df:T(M)\to T(N)$. 
It is a monomorphism of Euclidean bundles and induces a unique Bogoljubov morphism 
$\CCl_{Df}: \CCl^{(+)}(T(M))\to \CCl^{(+)}(T(N))$ of the Clifford bundles that is actually an isomorphism of $\CCl^{(+)}(T(M))$ with subbundle  $\CCl^{(+)}(T(f(M)))$, the Clifford bundle of the submanifold 
$f(M)$. (From this we see that the subalgebra $\CCl^{(+)}(f(M))\subset \CCl^{(+)}(N)$ of sections of the Clifford bundle of $N$ with support in $f(M)$ is naturally isomorphic with the algebra $\CCl^{(+)}(M)$ of sections of the Clifford bundle of $M$. 
Since the restriction to $f(M)$ is a natural epimorphism $\rho: \CCl^{(+)}(N)\to \CCl^{(+)}(f(M))$, ($\rho$ acts on Clifford fields by multiplication with the characteristic function of $f(M)$), there is a natural unital epimorphism of algebras $\psi: \CCl^{(+)}(N)\to \CCl^{(+)}(M)$ that becomes an isomorphism when restricted to $\CCl^{(+)}(f(M))$.) This isomorphism can be used to ``transfer'' the irreducible Clifford action of $\CCl^{(+)}(T(f(M)))$ on the bundle $S(f(M)):=S(N)|_{f(M)}$ to an irreducible action of $\CCl^{(+)}(T(M))$ and, since the bundle $f^\bullet(S(N))=f^\bullet(S(f(M)))$ is naturally isomorphic to $S(f(M))$, the bundle 
$f^\bullet(S(N))$ becomes an irreducible complex spinor bundle on $M$. 
By a similar argument, $f^\bullet(S(N))$ comes equipped with a spinorial charge conjugation 
$f^\bullet(C_N)$ obtained by ``pull-back'' of (the restriction to $S(f(M))$ of) $C_N$ through the isomorphism 
$f^\bullet(S(N))\simeq S(f(M))$. 

\medskip

We say that $f$ is \emph{spin-preserving} if the spin structure $[(f^\bullet(S(N)),f^\bullet(C_N))]$ determined by $f^\bullet(S(N))$ with spinorial charge conjugation $f^\bullet(C_N)$ coincides with the spin structure of $M$ i.e.~if there exists an isomorphism of Hermitian bundles $U: f^\bullet(S(N))\to S(M)$ that intertwines the charge conjugations: $ U\circ f^\bullet(C_N)= C_M\circ U$ and the Clifford actions. 
Note that if $f$ is orientation-preserving, the isomorphism $U$ also intertwines the grading operators of the spinor bundles.

Let us denote by $\Mf_d$-spin the subcategory of spin and orientation-preserving maps in $\Mf_d$. 

The following result, that we report for completeness, is certainly well-known although we could not find any suitable reference.  
Note that AS-$\Sf_{Ired}$ denotes the full subcategory of $\Sf_{Ired}$ whose objects are  direct sums of irreducible Abelian spectral triples. 

\begin{proposition}\label{prop: iso-R}
Let $M, N$ be two compact orientable Riemannian spin-manifolds in the category $\Mf$.  
If $f:M \to N$ is a spin-preserving isomorphism of Riemannian manifolds, the spectral triples $(\As_M,\H_M,D_M)$ and 
$(\As_N,\H_N,D_N)$ are isomorphic in the category \hbox{AS-$\Sf_{Ired}$}. 
\end{proposition}
\begin{proof}
The pull-back $\phi:=f^\bullet$ is a $*$-isomorphism $\phi: \As_N\to \As_M$ of pre-C*-algebras. 

\medskip

Consider the ``pull-back of spinor fields'' given by the invertible map $\Psi:= \sigma \mapsto \sigma \circ f$, for all $\sigma\in\H_N$.
Since $f$ is an orientation-preserving Riemannian isometry, it leaves invariant the volume forms $f^\bullet(\mu_N)=\mu_M$ and so 
we obtain 
\begin{equation*}
\int_M \ip{\Psi(\sigma)(x)}{\Psi(\tau)(x)}\, d\mu_M(x)=\int_N\ip{\sigma(y)}{\tau(y)}\, d\mu_N(y)
\end{equation*}
that implies that the map $\Psi: \H_N \to \text{L}^2(M,f^\bullet(S(N)))=:\H^\bullet$ is a unitary operator.

Since $f^\bullet(S(N))$ is a Hermitian bundle over $M$, $\H^\bullet$ carries a natural representation $\pi^\bullet$ of the algebra 
$\As_M$ given by pointwise multiplication. $\Psi$ intertwines $\pi_N$ and $\pi^\bullet\circ\phi$, 
i.e.~$\Psi(\pi_N(a)\sigma)=\pi^\bullet(\phi(a))\Psi(\sigma)$ for $a\in\As_N$ and $\sigma\in\H_N$.

\medskip

Let $U: f^\bullet(S(N))\to S(M)$ be a (noncanonical) isomorphism of Hermitian bundles induced by the spin-preserving condition on 
$f$. 
Since we know that $U$ is unitary on the fibers, we have 
$\int_M\ip{U\sigma(p)}{U\tau(p)}_{S_p(M)}\,\text{d}\mu_M(p)=
\int_M\ip{\sigma(p)}{\tau(p)}_{f^\bullet_p(S(N))}\,\text{d}\mu_M(p)$, for all 
\hbox{$\sigma,\tau\in \Gamma^\infty(f^\bullet(S(N)))$}. Hence $U$ uniquely extends to a unitary map 
$\Theta_U:\H^\bullet\to \H_M$. Note that $\Theta_U$ is $\As_M$-linear: 
$\Theta_U(a\cdot \sigma)=a\cdot \Theta_U(\sigma)$, for $a\in \As_M$ and $\sigma\in \H^\bullet$.

\medskip

Now it is not difficult to check that the pair $(\phi,\Theta_U\circ \Psi)$ is an isomorphism in the category $\Sf_{Ired}$ from the spectral triple $(\As_N,\H_N,D_N)$ to $(\As_M,\H_M,D_M)$. 
\end{proof}

\begin{lemma}\label{lem: full}
If $M$ and $N$ are two orientable compact Riemannian spin-manifolds in the category $\Mf$ and $(u,U)$ is an isomorphism from  $(\As_N,\H_N,D_N)$ to $(\As_M,\H_M,D_M)$ in the category AS-$\Sf_{Ire}$, then there is a spin-preserving orientation-preserving Riemannian isometry (metric isometry) $f: M\to N$ such that $f^\bullet=u$. 
\end{lemma}
\begin{proof}
The map $u: \As_N\to\As_M$ naturally extends to a $*$-isomorphisms of C*-algebras and by Gel'fand theorem there exists a homeomorphism $f: M \to N$ such that $f^\bullet=u$. Since $f^\bullet$ maps smooth functions onto smooth functions, $f$ is a diffeomorphism. 

The filtered algebra $\Omega_M(\As_M)$ (respectively $\Omega_N(\As_N)$) coincides with the filtered algebra of smooth sections of the Clifford bundle $\CCl^{(+)}(T(M))$ (respectively $\CCl^{(+)}(T(N))$) and the 
map $\Ad_U: \Omega_{D_N}(\As_N)\to\Omega_{D_M}(\As_M)$ is a filtered isomorphisms (extending $f^\bullet$). 
Therefore its restricition $\Ad_U: \Omega^1_{D_N}(\As_N)\to\Omega^1_{D_M}(\As_M)$ is an isomorphism between 
the Hermitian modules of sections of the complexification of the tangent bundles $T(M)$ and $T(N)$. 

From Serre-Swan theorem, $Df: T(M)\to T(N)$ is an isomorphism of Euclidean bundles which implies that $f$ is a Riemannian isometry. 

Since $\Ad_U(J_N)=J_M$ and $\Ad_U(\Gamma_N)=\Gamma_M$, $f$ is orientation- and spin-preserving.
\end{proof}

We can now state the promised equivalence result. 
\begin{theorem}\label{th: eq}
The functor $\Cg$ is an equivalence between the category \hbox{$\Mf_d$-spin} and the category AS-$\Sf^0_{Ired}$.
\end{theorem}
\begin{proof}
First, we show that the functor $\Cg$ is an embedding of the category $\Mf_d$-spin into AS-$\Sf^0_{Ired}$, and thus it is faithful.

\medskip

Let $f: M \to N$ be a spin-preserving metric isometry in $\Mf_d$-spin.  
By Lemma~\ref{lem: smooth} \hbox{$f: M\to N$} is a smooth Riemannian isometry onto the closed submanifold $f(M)$, a union of connected components of $N$. 

\medskip

We denote by $\rho: \As_N\to \As_{f(M)}$ the restriction epimorphism. 

\medskip

The Hilbert space $\H_N=\text{L}^2(N,S(N))$ decomposes as the direct sum $\oplus_{j\in \pi^0(N)}\H_j$ of Hilbert spaces (one for each connected component $j\in\pi^0(N)$ of $N$) and the multiplication operator $P$ by the characteristic function $\chi_{f(M)}$ is the projection operator onto the subspace $\H_{f(M)}:=P(\H_N)=\oplus_{j\in\pi^0(f(M))}\H_j$ (cf.~\cite[Page~491]{FGV}).Note that, since the Dirac operator $D_N$ is ``local'' (i.e.~it preserves the support of the spinor fields), the subspace $\H_{f(M)}$ is invariant for $D_N$. In the same way, since $J_N$ and $\Gamma_N$ acts fiberwise, $\H_{f(M)}$ is invariant for the charge conjugation and grading of $N$.   

\medskip

Defining $D_{f(M)}:=P\circ D_N\circ P$, $J_{f(M)}:=P\circ J_N\circ P$ and $\Gamma_{f(M)}:=P\circ \Gamma_N\circ P$, 
it is immediate that $(\As_{f(M)},\H_{f(M)},D_{f(M)})$ is a real (even) spectral triple and it follows that the ``restriction'' map $P: \H_N\to\H_{f(M)}$ satisfies $\forall a\in \As_N$, $\sigma\in \H_N$, $P(a\sigma)=\rho(a)P(\sigma)$, $P\circ D_N=D_{f(M)}\circ P$, $P\circ J_N=J_{f(M)}\circ P$, $P\circ \Gamma_N=\Gamma_{f(M)}\circ P$. 
This means that the pair $(\rho,P)$ is a morphism in the category $\Sf_{Ired}$ from  $(\As_N,\H_N,D_N)$ to the triple $(\As_{f(M)},\H_{f(M)},D_{f(M)})$, which is nothing but the spectral triple obtained from the manifold $f(M)$. 
By Proposition~\ref{prop: iso-R}, there exists an isomorphism from $(\As_{f(M)},\H_{f(M)},D_{f(M)})$ to $(\As_N,\H_N,D_N)$ in the category $\Sf_{Ired}$, and the conclusion follows by composition with the previous $(\rho,P)$. 

\medskip

Next, we show that the identity functor is an inclusion of the category AS-$\Sf^0_{Ired}$ into the category AS-$\Sf^m_d$, which implies that
$\Cg$ is representative. 

\medskip

Let $\phi: (\As_1,\H_1,D_1) \to (\As_2,\H_2,D_2)$ be an isomorphism in the category AS-$\Sf^0_{Ired}$. 
By the reconstruction theorem~\ref{th: Co-Re}, there are two manifolds $M$ and $N$ in the category $\Mf$ such that 
$(\As_N,\H_N,D_N)$ is isomorphic to $(\As_1,\H_1,D_1)$ and $(\As_M,\H_M,D_M)$ is isomorphic to $(\As_2,\H_2,D_2)$ with isomorphisms $(\phi_N,U_N)$ and $(\phi_M,U_M)$, respectively, in the category $\Sf_{Ired}$. 

By lemma~\ref{lem: full}, $\phi_M\circ\phi\circ\phi^{-1}_N\in \Sf^0_{Ired}$ is the image under $\Cg$ of a spin-preserving Riemannian isometry $f$ that (for manifolds of the same dimension) is a metric isometry in $\Mf_d$.

Since $\phi_M,\phi_N$ are isomorphisms in  AS-$\Sf^0_{Ired}$ and hence, by remark~\ref{rm: eq}, 
isomorphisms also in AS-$\Sf^m_d$, it follows that $\phi=\phi_M^{-1}\circ\Cg(f)\circ\phi_N\in\ $AS-$\Sf^m_d$.

\medskip

Finally, we show that $\Cg$ is full. 

\medskip

Let $M$ and $N$ be manifolds in the category $\Mf_d$-spin and let 
$\phi: \Cg(N)\to \Cg(M)$ be a morphism in the category $\Sf^0_{Ired}$. 
Since we proved above that AS-$\Sf^0_{Ired}$ is a subcategory of AS-$\Sf^m_d$, we have that 
$\phi$ is a morphism in the category $\Sf^m_d$ and from remark~\ref{rk: co-re} there exists a metric isometry $f: M\to N$ in the category $\Mf_d$ such that $\Cg(f)=\phi$. 
Since $\phi$ defines an isomorphism between $\Cg(f(M))$ and $\Cg(M)$ in $\Sf^0_{Ired}$ then, by lemma~\ref{lem: full}, 
$f: M \to f(M)$ is (orientation and) spin-preserving and we are done.
\end{proof} 

Let us summarize the categorical ``relations'' now available with the commutative diagram of functors (please refer to the appendix~\ref{a: st} for a list of the categories involved)
\begin{equation*}
\xymatrix{
{\text{AS-}\Sf^0_{Ired}}\ar@{^{(}->}[rr]&&{\text{AS-}\Sf^m_d} \ar@{^{(}->}[r]_{} &{\text{AS-}\Sf^m}
\\
	& {\text{AS-}\Sf^m_d\text{-spin}}\ar@{_{(}.>}[lu]\ar@{^{(}.>}[ur] & & 
\\
	{\Mf_d\text{-spin}} \ar@{.>>}[ur]_{\Cg}	\ar@{^{(}->}[rr]_{} \ar@{->}[uu]^{\Cg} &
&	{\Mf_d}\ar@{^{(}->}[r]_{}  \ar@{->}[uu]_{\Cg}		& {\Mf,}\ar@{->}[uu]_{\Cg}
}
\end{equation*}
where $\text{AS-}\Sf^m_d\text{-spin}:=\Cg(\Mf_d\text{-spin})$.
The left and right vertical inclusion functors correspond respectively to the embedding described at the beginning of the proof in theorem~\ref{th: eq} and to the 
anti-equivalence in theorem~\ref{th: co-re}; 
the horizontal top-left arrow is the inclusion functor described in the course of the proof of theorem~\ref{th: eq}.  

\bigskip

Loosely speaking, one would expect a similar structure to carry over to the general non-commutative setting, relating 
subcategories of ``spin-preserving'' morphisms in $\Sf^m$ and ``metric-preserving'' morphisms in $\Sf^0_{Ire}$. 
However, in general things might be more complicated. For the time being, we just mention the following result, omitting the (easy) details of the proof.

\begin{proposition}
Let $(\As_1,\H_1,D_1)\xrightarrow{(\phi,\Phi)}(\As_2,\H_2,D_2)$ be a morphism of the spectral triples in the category $\Sf$, where 
$\Phi$ is a coisometry. Then 
\begin{equation*}
d_{D_1}(\omega_1 \circ \phi,\omega_2 \circ \phi)\leq d_{D_2}(\omega_1,\omega_2), \quad \forall \omega_1, \omega_2 \in \Ss(\As_2).
\end{equation*}
\end{proposition}

\bigskip

We have discussed only the case of spin-manifolds.
We also expect analogous statements to hold true for spin$^c$ manifolds. 

\section{Final Comments}

The main result presented in this paper is a reformulation of the Gel'fand-Na\u\i mark duality in the light of Connes' reconstruction theorem for spin Riemannian manifolds. 
It seems to us that the functoriality of such correspondence has some intriguing appeal  and one could ask to which extent it is possible to ``lift'' it to some of the other main objects entering the scene, notably the Dirac operators. 
This issue is presently under investigation. 

From the perspective of this work, the use of the spin structure has been only instrumental in recasting Gel'fand-Na\u\i mark theorem in the light of the Connes' reconstruction theorem, and actually it might appear as an unnecessary complication: it introduces some redundancy in the main result and, when incorporated tout-court in the setup, it does not lead to a genuine categorical anti-equivalence.

This might suggest that in a successive step one could try to get rid of such a structure, thus obtaining a different kind of categorical duality between a metric category of (isometries of) Riemannian manifolds and suitable categories of spectral data (for example considering spectral triples arising from the signature Dirac operator in place of those arising from the usual Atiyah-Singer Dirac operator). 
Although several variants of morphisms can be introduced between spectral triples (see~\cite[Section~4.1]{BCL0} for details), corresponding to the ``rigidity'' imposed on the maps between manifolds (totally geodesic isometries, Riemannian isometries, \dots), this line of thought does not require significant structural modifications in the definitions of morphisms for the categories of spectral geometries involved (as a pair of maps at the algebra and the Hilbert space level) and will be pursued elsewhere (see~\cite{Be} for more details).

The actual construction of functors (and dualities) from categories of spin Riemannian manifolds (with different dimensions) to ``suitable'' categories of spectral triples (of the Atiyah-Singer ``type'') is a more interesting goal whose main obstruction is the lack of a sufficiently general notion of pull-back of spinor fields. In order to solve this problem it will be necessary to construct ``relational categories'' of spectral triples, via ``spectral congruences'' and/or ``spectral spans'' following the lines already announced in the seminar slides~\cite{Be}. We will return to these topics in forthcoming papers. 

Finally, we observe that
after the first draft of the present manuscript had been already submitted for publication, a few other notable works on categorical non-commutative geometry appeared. 
In particular we mention the recent paper by B.Mesland~\cite{Me} where, in the setting of KK-theory, a notion of Morita morphism between spectral triples is defined via very specific Connes' correspondence bimodules. 

\appendix 

\section{List of Categories of Spectral Triples}\label{a: st}

\begin{itemize}
\item[$\Sf$] the category whose objects are spectral triples $(\As,\H,D)$ and 
whose morphisms from $(\As,\H,D)$ to $(\As',\H',D')$ are the pairs $(\phi,\Phi)$, 
where $\phi:\As\to \As'$ is a unital \hbox{$*$-homomorphism} and $\Phi:\H\to \H'$ is a 
(bounded) linear map such that $\Phi(ax)=\phi(a)\Phi(x)$, 
$\Phi(D\xi)=D'(\Phi \xi)$, 
for all $a\in \As$, $x\in \H$ and $\xi \in {\rm Dom}(D)$. 
\item[$\Sf_r$]
the category  consisting of real spectral triples $(\As,\H,D,J)$ with morphisms 
$(\phi,\Phi)$ defined as for $\Sf$ but in addition intertwining the real 
structures, i.e.  $\Phi J=J'\Phi$. 
\item[$\Sf_e$]
the category consisting of even spectral triples $(\As,\H,D,\Gamma)$ and morphisms 
$(\phi,\Phi)$ of triples such that, in addition, $\Phi \Gamma=\Gamma'\Phi$. 

\item[$\Sf_{re}$] defined as the category consisting of real, even spectral 
triples and morphisms of triples intertwining both the real structure and the 
grading.

\item[$\Sf_I$] the subcategory of $\Sf$ with morphism $(\phi,\Phi)$, with $\phi$ 
surjective and $\Phi$ co-isometric. Similarly define
$\Sf_{I\bullet}$ as a subcategory of $\Sf_{\bullet}$, where $\bullet$ stands for 
any of the choices $r,e,re$. 

\item[$\Sf^m$] the category of spectral triples from $\Sf$ with arrows given by unital $*$-epimorphisms $\phi:\As\to \As'$ such that 
$d_D(\phi^\bullet(\omega_1),\phi^\bullet(\omega_2))=d_{D'}(\omega_1,\omega_2)$ for all pure states $\omega_1,\omega_2$ of 
$\As'$. 

\item[$\Sf^0$] the category with the same objects of $\Sf$ obtained by quotienting 
the arrows in $\Sf$ via the equivalence relation 
$(\phi_1,\Phi_1)\simeq(\phi_2,\Phi_2)$ iff $\phi_1=\phi_2$. 
Similarly $\Sf^0_{Ire}$ is obtained from $\Sf_{Ire}$ as a quotient via the same 
equivalence relation. 

\item[AS-$\Sf$] the full subcategory of $\Sf$ with objects the direct sums of 
Atiyah-Singer spectral triples, i.e. those triples that satisfy 
the conditions of Connes' reconstruction theorem. 
\item[AS-$\Sf_{Ire}$] is the full subcategory of $\Sf_{Ire}$ with objects the
direct sums of Atiyah-Singer spectral triples.

\item[AS-$\Sf^0_{Ire}$] 
is the full subategory of $\Sf^0_{Ire}$ with objects the direct sums of Atiyah-Singer spectral triples. 

\end{itemize} 
The suffix $d$ (as in $\Sf_d$, AS-$\Sf_d$, AS-$\Sf_{Ired}$ or $\Sf^m_d$) indicates the  
subcategory consisting only of morphisms between triples of the same 
dimension.

\bigskip

\emph{Notes and acknowledgments.} 
\label{sec: ac}
\\
We acknowledge the support provided by the Thai Research Fund through the ``Career Development Grant'' n.~RSA4780022. 

The first arXived version of this work (18 December 2008) was submitted to the Bullettin of the Korean Mathematical Society and published after revision on 11 March 2011. The following is a version of the printed paper prepared only for upload to the arXiv. 

Finally it is a pleasure to thank Starbucks Coffee at the $1^\text{st}$ floor of Emporium Tower Sukhumvit for providing the kind relaxing environment where essentially all of this work has been written in the first half of 2005.

\end{document}